\documentclass[11pt]{amsart}

\usepackage[pagebackref=true,colorlinks,linkcolor=blue,citecolor=magenta]{hyperref}
\usepackage{amsmath,amssymb,latexsym} % ams auxiliaries.
\usepackage{graphicx}
\usepackage{fullpage}
\usepackage[square, numbers]{natbib}   % Arguably a nicer bibliography.
\usepackage{comment}
\usepackage{bm}

\newtheorem{theorem}{Theorem} % If subsection, get 3.1.2, etc.

\theoremstyle{definition}
\newtheorem{definition}{Definition}

\newtheorem{question}{Question}

\title{On The Number of Irreducible FAT Colorings}
\author{Saeed Shaebani}
\address{S. Shaebani, 
	School of Mathematics and Computer Science,
	Damghan University, Damghan, Iran.}
\email{shaebani@du.ac.ir}

\begin{document}
\begin{abstract}
\noindent 
A vertex coloring of a graph $G$ with nonempty color classes $V_1,V_2,\dots,V_k$ is called a \emph{FAT $k$-coloring} if there exist real numbers $\alpha,\beta\in[0,1]$ such that for every vertex $v$ and every color class $V_i \in \left\{ V_1,V_2,\dots,V_k \right\} $ we have
$$
\bigl| N(v) \cap V_i \bigr|=
\begin{cases}
	\alpha\deg(v) & \text{if } v\notin V_i,\\[4pt]
	\beta\deg(v)  & \text{if } v\in V_i .
\end{cases}
$$

\noindent
The FAT coloring concept was originally proposed and thoroughly studied by Beers and Mulas.
The set of all FAT colorings of a graph is naturally ordered by the coarsening relation, in which finer partitions are larger in the order. The maximal elements of this poset, called \emph{irreducible FAT colorings}, form a generating set: every FAT coloring of the graph can be obtained by merging color classes of some irreducible one. Beers and Mulas raised the compelling question whether, for every positive integer $s$, there exists a graph that admits exactly $s$ irreducible FAT colorings. In this paper we settle this question affirmatively by exhibiting, for any given $s$, a graph possessing precisely $s$ such colorings.

\noindent {\bf Keywords:}\ {Vertex Degrees, Fair And Tolerant Vertex Coloring, Irreducible FAT Coloring.}\\
{\bf Mathematics Subject Classification: 05C15, 05C07.}
\end{abstract}

\maketitle
\section{{\bf Introduction}}

The present work is concerned with parameters that arise in the context of graph coloring. Unless the brackets serve as a wrapper for a bibliographic label, the symbol $[n]$ (with $n$ a positive integer) shall throughout this article be understood as the initial segment consisting of the first $n$ natural numbers $\{1,2,\dots,n\}$.

\noindent
The scope of this paper is restricted to simple graphs with finite and nonempty vertex sets. Throughout, the notation is consistent with that of the foundational texts \cite{Bollobas1998}, \cite{BondyMurty2008}, \cite{Diestel2017}, and \cite{West2001}.

\noindent
The \emph{order} of a graph $G$ is the cardinality of its vertex set, denoted by $|V(G)|$, and its \emph{size} is the number of edges, written $|E(G)|$. Given a vertex $v$ of $G$, its \emph{neighborhood} $N_G(v)$ is the collection of all vertices adjacent to $v$; members of this collection are called \emph{neighbors} of $v$. The \emph{degree} of $v$, denoted $\deg_G(v)$, is defined as the cardinality $|N_G(v)|$. When the graph $G$ is apparent from the context, we drop the subscript and simply write $N(v)$ and $\deg(v)$ for the neighborhood and the degree, respectively.

\noindent
Take a graph $G$, a vertex $v$, and a subset $S\subseteq V(G)$. We write $e(v,S)$ for the count of vertices in $S$ that are adjacent to $v$. In precise terms,
$$e(v,S) := \bigl| \{ u\in S \colon uv \in E(G) \} \bigr|.$$

\noindent
For any graph $G$ and any subset $S\subseteq V(G)$, the \emph{induced subgraph} on $S$, denoted $G[S]$, is the graph whose vertex set is $S$ and in which two distinct vertices are joined by an edge if and only if they are adjacent in $G$.

\noindent
By a \emph{complete graph} on $n$ vertices, denoted $K_n$, we mean a graph of order $n$ in which every two distinct vertices are adjacent.

\noindent
A \emph{clique} of a graph $G$ is a nonempty subset $S\subseteq V(G)$ whose induced subgraph $G[S]$ is complete, meaning that every two distinct vertices in $S$ are adjacent. The greatest possible size of a clique in $G$ is termed the \emph{clique number} and denoted by $\omega(G)$.

In the context of a graph $G$, a \emph{vertex coloring} is defined to be any function $c  \colon V(G) \to C$ mapping the vertex set $V(G)$ into a set $C$ whose elements are called \emph{colors}; the set $C$ itself is referred to as the \emph{color collection}. When $c(v)=r$, we say that $v$ is \emph{assigned} the color $r$.

\noindent
Given a coloring $c\colon V(G)\to C$ and a color $r\in C$, the fiber $c^{-1}(r)=\{v\in V(G) \colon c(v)=r\}$ is called the \emph{color class} of $r$; it is exactly the collection of vertices assigned the color $r$.

A \emph{partition} of a set $S$ stands for a collection of some nonempty and pairwise disjoint subsets of $S$ whose union is $S$.

Let $G$ be a graph. Any partition of $V(G)$ can be expressed as the collection of nonempty color classes arising from some vertex coloring of $G$. Conversely, the nonempty fibers of a vertex coloring $c\colon V(G)\to C$ form a partition of the vertex set $V(G)$. The structure of this partition is independent of the specific labels chosen for the colors; it simply groups vertices into pairwise disjoint blocks according to the color they receive. We conclude that if our only concern is which vertices receive the same color, and the names of the colors play no role, then all that matters is the partition of the vertex set into color classes, irrespective of the specific labels attached to those classes. Thus, a vertex coloring of a graph $G$ is nothing other than a partition of $V(G)$ into $k$ nonempty blocks $V_1,\ldots,V_k$, each called a color class.

Consider a graph $G$ and a vertex coloring of $G$ with nonempty color classes $V_1,V_2,\dots,V_k$. This coloring is said to be a \emph{Fair And Tolerant $k$-coloring} of $G$ if one can find real numbers $\alpha,\beta\in[0,1]$ with the property that for each vertex $v$ and each color class $V_i$,
$$
e(v,V_i)=
\begin{cases}
	\alpha\deg(v) & \text{if } v\notin V_i,\\[4pt]
	\beta\deg(v)  & \text{if } v\in V_i,
\end{cases}
$$
where, as usual, $e(v,V_i)$ denotes the number of neighbors of $v$ that belong to $V_i$.

\medskip

\noindent
The Fair And Tolerant vertex coloring concept was originally proposed and thoroughly studied by Beers and Mulas in \cite{beers20252, beers2026}. To keep the terminology concise, they coined the abbreviated phrase ``\emph{FAT coloring}'', and also employed the term ``\emph{FAT $k$-coloring}'' to designate a Fair And Tolerant vertex coloring with $k$ colors.

\medskip

Let $G$ be a graph possessing a vertex $v$ of positive degree. If $G$ admits a FAT $k$-coloring with nonempty color classes $V_1,\dots,V_k$ and associated parameters $\alpha,\beta\in[0,1]$, then evaluating the degree of $v$ yields
$$
\deg(v)=\sum_{i=1}^{k} e(v,V_i) = \beta\deg(v)+(k-1)\alpha\deg(v).
$$
Since $\deg(v)>0$, dividing by $\deg(v)$ gives the fundamental identity
$$
\beta+(k-1)\alpha=1,
$$
which links $k$, $\alpha$, and $\beta$ as noted in \cite{beers20252}.

\medskip

Every graph $G$ trivially possesses a FAT $1$-coloring: setting $\alpha=0$, $\beta=1$ and letting $V(G)$ be the unique color class yields $e(v,V(G))=\deg(v)=\beta\deg(v)$ for all $v$, satisfying the definition. Consequently, the set
$$
\mathcal{A}(G)=\{k\in\mathbb{N} \ \colon \ \text{$G$ admits a FAT $k$-coloring}\}
$$
is nonempty. Since in any FAT coloring the color classes are nonempty and partition $V(G)$, the number of classes cannot exceed $|V(G)|$; thus $\mathcal{A}(G)$ is bounded above. Hence $\mathcal{A}(G)$ possesses a greatest element; this integer is termed the \emph{FAT chromatic number} of $G$ and is denoted by $\chi^{\mathrm{FAT}}(G)$. In symbols,
$$
\chi^{\mathrm{FAT}}(G):=\max\{k\in\mathbb{N} \ \colon \ G \text{ admits a FAT } k\text{-coloring}\}.
$$
This notion was originated and investigated by Beers and Mulas~\cite{beers20252, beers2026}.

\maketitle
\section{\bf The Coarsening Order and Irreducible FAT Colorings}

Let $X = \left\{ X_1 , X_2 , \dots , X_m \right\}$ and
$Y = \left\{ Y_1 , Y_2 , \dots , Y_n\right\}$ be two partitions of the same set $V$. We say that $X$ is \emph{finer} than $Y$ (or eaquivaletly, $Y$ is \emph{coarser} than $X$) whenever each element of $Y$ is a union of some
elements in $X$. More precisely, $X$ is said to be finer than $Y$ if for each $Y_j \in \left\{ Y_1 , Y_2 , \dots , Y_n\right\}$ there exist some $X_{i_1} , X_{i_2}, \dots , X_{i_t}$ for which
$Y_j = X_{i_1} \cup X_{i_2} \cup \dots \cup X_{i_t}$. Also, $X$ is called \emph{strictly finer} than $Y$ whenever $X \neq Y$ and the partition $X$ is finer than $Y$. Besides, $Y$ is said to be \emph{strictly coarser} than $X$ whenever the partition $X$ is strictly finer than the partition $Y$.

As stated in \cite{beers20252, beers2026}, a useful feature of FAT colorings is that appropriate merging color classes preserves the FAT conditions; producing a coarser coloring whose parameters scale in a predictable way.

\noindent
Let $V_1, V_2, \dots,V_k$ be the color classes of a FAT $k$-coloring with parameters $\alpha$ and $\beta$. For any positive divisor $\ell$ of $k$ with $\ell \neq 1$, define new color classes $W_1,\dots,W_\ell$ by taking the union of $\frac{k}{\ell}$ of the original color classes each. Now, let us take some vertex $v \in W_j$, and assume that $v \in V_i$ in the original coloring.
For any block $W_h$ with $h \neq j$, the number of neighbors of $v$ in $W_h$ is the sum of its neighbors in the $\frac{k}{\ell}$ constituent classes of the original coloring. Because $v$ has exactly $\alpha \deg(v)$ neighbors in each of those color classes, the total is $\frac{k}{\ell} \alpha \deg(v)$. Hence $\alpha' =  \frac{k}{\ell}\,\alpha$ serves as the fairness parameter for
the coarser coloring \cite{beers20252}. Also, the set of neighbors of $v$ in $W_j$ is the union of its neighbors in the $\frac{k}{\ell}-1$ constituent classes of the original coloring together with its neighbors in
$V_i$. Because $v$ has exactly $\alpha \deg(v)$ neighbors in each of those $\frac{k}{\ell}-1$ classes and it has exactly $\beta \deg(v)$ neighbors in $V_i$, the total is
$\Bigl( \beta + \left(  \frac{k}{\ell}-1 \right) \alpha \Bigr) \deg(v)$. So, $\beta' = \beta + \left( \frac{k}{\ell} - 1 \right) \alpha$ serves as the tolerance parameter for the coarser coloring \cite{beers20252}.
Therefore, the resulting new $\ell$-coloring is also FAT \cite{beers20252} with parameters
	$$\begin{array}{lcccr}
		\alpha' = \frac{k}{\ell}\,\alpha & & {\rm and} & & \beta' = \beta + \left( \frac{k}{\ell} - 1 \right) \alpha .
	\end{array}$$
This interesting property of FAT colorings, motivates the following partial order introduced by Beers and Mulas \cite{beers20252}.

\begin{definition}
	Let $G$ be a graph. Also, let $X = \left\{ X_1 , X_2 , \dots , X_m \right\}$ and
	$Y = \left\{ Y_1 , Y_2 , \dots , Y_n\right\}$ be two partitions of the same set $V(G)$, each establishing some FAT coloring of $G$. We write $X \preceq Y$ whenever $X$ is coarser than $Y$.
	Moreover, $X \prec Y$  indicates that $X \preceq Y$ and $X \neq Y$. The relation $\preceq$, which is called \emph{the coarsening order}, establishes a partial order on the set of all FAT colorings of $G$.
\end{definition}

The unique FAT $1$-coloring in which all vertices lie in one unified class is the minimum element of this poset. The maximal elements are particularly important because every FAT coloring can be obtained by merging classes of some maximal one. Accordingly, the maximal ones deserve a specific title, and they are referred to as \emph{irreducible FAT colorings} \cite{beers20252}.

\begin{definition}
	A FAT coloring is said to be \emph{irreducible} if it is maximal with respect to $\preceq$; that is, it is not strictly coarser than any other FAT coloring. A FAT coloring that is not irreducible is called \emph{reducible}.
\end{definition}

The irreducible FAT colorings form a “generating set” for all FAT colorings of the graph, that is, every FAT coloring can be derived from some irreducible FAT coloring by repeatedly merging color classes.

The complete classification of irreducible FAT colorings constitutes a natural structural problem, and it has been carried out for several families of graphs in the original works~\cite{beers20252, beers2026} by means of combinatorial and spectral techniques.

\section{\bf The Main Result}

This section is devoted to the main result of the paper. In the original work \cite{beers20252}, Beers and Mulas raised the following compelling question.

\begin{question}
	For a given positive integer $s$, does there exist a graph that admits exactly $s$ irreducible FAT colorings?
\end{question}

The central purpose of this paper is to settle the question posed above; the following theorem yields an affirmative answer.

\begin{theorem}
	For any positive integer $s$, there exists a graph possessing precisely	$s$ irreducible FAT colorings.
\end{theorem}

\begin{proof}
	The proof is divided into three cases: $s=1$, $s=2$, and $s\geq 3$.

	{\bfseries The Case $\bm{s=1}$)} Let $n$ be a positive integer greater than $1$.
	The set of all singletons $\big\{ \{v\} \ \colon \ v\in V \left( K_n \right) \big\}$
	provides the set of color classes of an irreducible FAT $n$-coloring (with parameters $\alpha = \frac{1}{n-1}$ and $\beta = 0$) for the complete graph on $n$ vertices $K_n$.
	
	\noindent
	Since $\big\{ \{v\} \ \colon \ v\in V \left( K_n \right) \big\}$ is the finest partition of $V \left( K_n \right)$ and it is also strictly finer than every other FAT coloring of $K_n$, we conclude that it is the lonely irreducible FAT coloring of $K_n$.
	Thus, $K_n$ has exactly one irreducible FAT coloring.
	
	{\bfseries The Case $\bm{s=2}$)} Let $C_6$ be the cycle over six vertices $v_0 , v_1 , \dots , v_5$ whose edge set equals
	$\left\{ v_0 v_1 , v_1 v_2 , v_2 v_3 , v_3 v_4 , v_4 v_5 , v_5 v_0 \right\}$.
	First, let us introduce three particular FAT colorings of $C_6$, as follows:
	\begin{itemize}
		\item a FAT $1$-coloring of $C_6$, say $\varphi _1$, whose lonely color class is
		$\varphi_1 ^{(1)} := V(C_6)$, with the corresponding parameters $\alpha _ {\varphi_1} = 0$ and $\beta _{\varphi _1} = 1$.
		\item a FAT $2$-coloring of $C_6$, say $\varphi _2$, whose color classes are
		$\varphi_2 ^{(1)} := \left\{ v_0 , v_2 , v_4 \right\}$ and $\varphi_2 ^{(2)} := \left\{ v_1 , v_3 , v_5 \right\}$, with the corresponding parameters $\alpha _ {\varphi_2} = 1$ and $\beta _{\varphi _2} = 0$.
		\item a FAT $3$-coloring of $C_6$, say $\varphi _3$, whose color classes are
		$\varphi_3 ^{(1)} := \left\{ v_0 , v_3 \right\}$ and $\varphi_3 ^{(2)} := \left\{ v_1 , v_4 \right\}$ and $\varphi_3 ^{(3)} := \left\{ v_2 , v_5 \right\}$. The corresponding parameters of this FAT coloring are $\alpha _ {\varphi_3} = \frac{1}{2}$ and $\beta _{\varphi _3} = 0$.
	\end{itemize}
	
	\noindent
	Now, let us regard an arbitrary FAT $k$-coloring of $C_6$, and denote it by $\mathcal{L}$. Also, let the corresponding parameters of $\mathcal{L}$ be
	$\alpha_{\mathcal{L}}$ and $\beta_{\mathcal{L}}$.
	Moreover, let the set of color classes of $\mathcal{L}$ be
	$\widetilde{\mathcal{L}} := \Big\{ \mathcal{L} ^{(1)} , \mathcal{L} ^{(2)} , \dots , \mathcal{L} ^{(k)}  \Big\}$. Besides, without loss of generality,
	we may assume that $v_0 \in \mathcal{L} ^{(1)}$.
	
	\noindent
	The graph $C_6$ is a $2$-regular and connected graph. Thus, on account of Theorem \cite{}, if some FAT $k$-coloring of $C_6$ exists then we must have
	$k \in \{1,2,3\}$.
	
	\noindent
	If $k=1$, then $\mathcal{L}$ has just one
	color class which is the whole vertex set of $C_6$. Therefore,
	$\widetilde{\mathcal{L}} = \left\{\varphi _1 ^{(1)}\right\}$ and $\mathcal{L}$ coincides with $\varphi_1$.
	
	\noindent
	If $k > 1$, then the connectedness of $C_6$ implies that $\alpha_{\mathcal{L}} > 0$.
	Since $k > 1$ and $\deg _{C_6} \left( v_0 \right) =2$, we must have
	$$2 \alpha_{\mathcal{L}} =   \alpha_{\mathcal{L}}  \times \deg _{C_6} \left( v_0 \right) = e \Big( v_0 , \mathcal{L}^{(2)} \Big) \in \{ 1,2 \} .$$
	Hence, $\alpha_{\mathcal{L}} \in \left\{ \frac{1}{2} , 1 \right\}$.
	Also, the relation
	$\beta_{\mathcal{L}} + (k-1) \alpha_{\mathcal{L}} = 1$
	implies that one of the following three items occurs:
	\vspace{0.4cm}
	\begin{itemize}
		\item $k=3$ and $\alpha_{\mathcal{L}} = \frac{1}{2}$ and $\beta_{\mathcal{L}} = 0$,
		\vspace{0.4cm}
		\item $k=2$ and $\alpha_{\mathcal{L}} = 1$ and $\beta_{\mathcal{L}} = 0$,
		\vspace{0.4cm}
		\item $k=2$ and $\alpha_{\mathcal{L}} = \beta_{\mathcal{L}} = \frac{1}{2}$.
	\end{itemize}
	\vspace{0.4cm}
	For the first item where $k=3$ and $\alpha_{\mathcal{L}} = \frac{1}{2}$ and $\beta_{\mathcal{L}} = 0$, every vertex of $C_6$ has exactly one neighbor in each of the
	color classes not containing that vertex.
	Thus, $\widetilde{\mathcal{L}}$ must be equal to
	$$\Bigl\{ \mathcal{L} ^{(1)} , \mathcal{L} ^{(2)} , \mathcal{L} ^{(3)}  \Bigr\} = \Bigl\{   \left\{ v_0 , v_3 \right\} , \left\{ v_1 , v_4 \right\} , \left\{ v_2 , v_5 \right\}  \Bigr\}.$$
	Hence $\widetilde{\mathcal{L}} = \Bigl\{  \varphi_3 ^{(1)} , \varphi_3 ^{(2)} , \varphi_3 ^{(3)}  \Bigr\}$; which implies that $\mathcal{L}$ coincides with $\varphi_3$.
	
	\vspace{0.4cm}
	\noindent
	For the second item where $k=2$ and $\alpha_{\mathcal{L}} = 1$ and $\beta_{\mathcal{L}} = 0$, the coloring $\mathcal{L}$ has two color classes, and every vertex of $C_6$ has two neighbors in the opposite color class.
	So, it requires that
	$$\widetilde{\mathcal{L}} =  \Bigl\{ \mathcal{L} ^{(1)} , \mathcal{L} ^{(2)} \Bigr\} = \Bigl\{   \left\{ v_0 , v_2 , v_4 \right\} , \left\{ v_1 , v_3 , v_5 \right\}  \Bigr\}.$$
	It follows that $\widetilde{\mathcal{L}} = \Bigl\{  \varphi_2 ^{(1)} , \varphi_2 ^{(2)} \Bigr\}$.
	Accordingly, $\mathcal{L}$ and $\varphi_2$ coincide.
	
	\vspace{0.4cm}
	\noindent
	Finally, it is turn to consider the item where
	$k=2$ and $\alpha_{\mathcal{L}} = \beta_{\mathcal{L}} = \frac{1}{2}$.
	
	\noindent
	For every vertex $v_i \in V \left( C_6 \right)$ and every color class $\mathcal{L} ^{(j)} \in \widetilde{\mathcal{L}} = \Bigl\{ \mathcal{L} ^{(1)} , \mathcal{L} ^{(2)} \Bigr\}$
	where $v_i \in \mathcal{L} ^{(j)}$, the vertex $v_i$ has exactly one neighbor in $\mathcal{L} ^{(j)}$ (because of $\beta_{\mathcal{L}} = \frac{1}{2}$)
	and exactly one neighbor in the opposite color class (because of $\alpha_{\mathcal{L}} = \frac{1}{2}$).
	Therefore, the set of all edges with one end in $\mathcal{L} ^{(1)}$ and the other end in $\mathcal{L} ^{(2)}$, establish a perfect matching in $C_6$.
	Thus, $\Bigl| \mathcal{L} ^{(1)} \Bigr| = \Bigl| \mathcal{L} ^{(2)} \Bigr| = 3$.
	Besides, the set of edges with both ends in $\mathcal{L} ^{(1)}$ establish a matching that uses all vertices of $\mathcal{L} ^{(1)}$.
	Hence, the number of vertices of $\mathcal{L} ^{(1)}$ is even,
	contradicting the fact that  $\Bigl| \mathcal{L} ^{(1)} \Bigr| = 3$.
	Accordingly, the item where $k=2$ and $\alpha_{\mathcal{L}} = \beta_{\mathcal{L}} = \frac{1}{2}$
	does not occur.
	
	\noindent
	It follows that the lonely FAT colorings of $C_6$ are just $\varphi_1$, $\varphi_2$, and $\varphi_3$.
	
	\noindent
	Since $\varphi_2$ and $\varphi_3$ are refinements of $\varphi_1$, one finds that
	$\varphi_1$ is not an irreducible FAT coloring of $C_6$.
	
	\noindent
	Since neither $\varphi_2$ nor $\varphi_3$ is a refinement of the other,
	both are irreducible FAT colorings.
	
	\noindent
	We conclude that $C_6$ has just two irreducible FAT colorings.
	
	{\bfseries The Case $\bm{s \geq 3$})}
	Let $G := K_{s-1} \cup K_1$ be a graph with vertex set
	$$V(G) := \{0,1,2, \dots , s-1\}$$
	and edge set
	$$E(G) := \{ ij \colon \ \ 1 \leq i < j \leq s-1 \}.$$
	Indeed, the graph $G$ is the disjoint union of the complete graph $K_{s-1}$ over vertices $1,2, \dots , s-1$ and the complete graph $K_1$ over the vertex $0$.
	
	\noindent
	We divide the set of all irreducible FAT colorings of $G$ into two parts.	
	The first part consists of all those irreducible FAT colorings for which $\{0\}$
	is a color class; and the second part consists of all those irreducible FAT colorings such that $\{0\}$ is not a color class.
	
	\noindent
	For the first part where $\{0\}$ is a color class, it is necessary to note that the partition of $V(G)$ into two sets $V_1$ and $V_2$ where
	$$\begin{array}{lcccr}
		V_1 := V \left( K_1 \right) = \{0\} &  & {\rm and} &  & V_2 := V \left( K_{s-1} \right) = \{1,2, \dots , s-1\}
	\end{array}$$
	provides a FAT $2$-coloring of $G$ with parameters
	$\alpha = 0$ and $\beta = 1$.
	
	\noindent
	We claim that this FAT coloring is the unique FAT coloring of $G$ for which $\{0\}$ is a color class.
	Suppose contrary to the claim that some FAT coloring of $G$ with color classes
	$W_1 , W_2 , \dots , W_t$
	and corresponding parameters $\alpha_W$ and $\beta_W$ exists in such a way that $W_1 = \{0\}$ and $t\geq 3$.
	One may assume without loss of generality that $1\in W_2$.
	Based on the definition of FAT colorings, since $1 \notin W_2$ and $1 \notin W_3$, we must have
	$$\begin{array}{lcccr}
		e \left( 1, W_1 \right) = \alpha_W \times \deg_G (1)  & & {\rm and} & & e \left( 1, W_3 \right) = \alpha_W \times \deg_G (1).
	\end{array}$$
	Thus, $e \left( 1, W_1 \right) = e \left( 1, W_3 \right) $.	
	But because of $e \left( 1, W_1 \right) = 0$ and $e \left( 1, W_3 \right) = \bigl| W_3 \bigr| $,
	one finds that $W_3 = \varnothing$; contradicting the fact that all color classes in a FAT coloring must be nonempty.
	We conclude that the partition $\left\{ V_1 , V_2 \right\}$ of $V(G)$ where
	$V_1 = V \left( K_1 \right) = \{0\} $ and $V_2 = V \left( K_{s-1} \right) = [s-1]$
	establishes the unique FAT coloring of $G$ subject to the condition that $\{0\}$ is a color class.
	
	\noindent
	Every partition $X= \left\{ X_1 , X_2 , \dots , X_k \right\}$ of $V(G)$ that is strictly finer than $\left\{V_1 , V_2\right\}$,
	must satisfy the condition $\{0\} \in X$; and therefore, it could never be a FAT coloring of $G$ because $\left\{V_1 , V_2\right\}$
	is the unique FAT coloring of $G$ that has $\{0\}$ as a color class.
	
	\noindent
	Accordingly, $\Bigl\{ \{0\} , \{1,2, \dots , s-1\} \Bigr\}$ is an irreducible FAT coloring of $G$, and it is the unique irreducible
	FAT coloring of $G$ subject to the condition that $\{0\}$ is a color class.
	
	\noindent
	For the second part where $\{0\}$ is not a color class, we initially introduce some $(s-1)$-colorings of $G$, say
	$f_1 , f_2 , \dots , f_{s-1}$.
	
	\noindent
	For each $1 \leq j \leq s-1$, the set of color classes of $f_j$ is defined as
	$\widetilde{f}_j = \Bigl\{   V^{(j)}_{1} , V^{(j)}_{2} , \dots , V^{(j)}_{s-1}   \Bigr\}$
	in such a way that
	$$
	V^{(j)}_{i}=
	\begin{cases}
		\{i\}    & \text{if } \ \ i \in [s-1] \setminus \{j\},\\[4pt]
		\{0,j\}  & \text{if } \ \ i=j .
	\end{cases}
	$$

	\vspace{0.2cm}
	
	\noindent
	Indeed, in the coloring $f_j$, the color class $V^{(j)}_{j}$ equals $\{0,j\}$, and all the other color classes are singletons $V^{(j)}_{i} = \{i\}$.
	
	\noindent
	Each of $f_1 , f_2 , \dots , f_{s-1}$ is a FAT $(s-1)$-coloring of $G$ with parameters $\alpha = \frac{1}{s-2}$ and $\beta = 0$.	
	Since the lonely strictly finer partition of
	$\left\{  	V^{(j)}_{1} , V^{(j)}_{2} , \dots , V^{(j)}_{s-1}  \right\}$
	is the collection of all singletons
	$\Bigl\{  \{0\} , \{1\} , \{2\} , \dots , \{s-1\} \Bigr\}$
	which is not a FAT coloring of $G$, we find that all of  $f_1 , f_2 , \dots , f_{s-1}$ are irreducible.
	
	For every other FAT coloring $\mathcal{M}$ of $G$ that has not $\{0\}$ as a color class and its set of color classes (say $\widetilde{\mathcal{M}}$)
	is other than all of $\widetilde{f}_1, \widetilde{f}_2, \dots , \widetilde{f}_{s-1}$, there exists a nonempty subset $W \subseteq \{1,2,\dots, s-1\}$
	such that $\{0\} \cup W$ is a color class of $\mathcal{M}$.
	Also, all other color classes of $\mathcal{M}$ would be some pairwise disjoint subsets of $[s-1] \setminus W$.
	Now, one of the following two conditions occurs:
	\begin{itemize}
		\item $W$ equals to some singleton $\{j\}$,
		\item $W = \left\{ j_1 , j_2 , \dots , j_t \right\}$ for some $t>1$ and some pairwise disjoint elements $j_1 , j_2 , \dots , j_t$.
	\end{itemize}
	In the former, $\{0,j\}$ would be a color class of $\mathcal{M}$.
	
	Now, since $\widetilde{\mathcal{M}} \neq \widetilde{f}_j$, we find that not all the other color classes of $\mathcal{M}$ could be singletons.
	Hence, $\widetilde{f}_j$ is strictly finer than $\widetilde{\mathcal{M}}$.
	
	\noindent
	In the latter, $\left\{ 0 , j_1 , j_2 , \dots , j_t \right\}$ is a color class of $\mathcal{M}$ and it equals
	$\left\{ 0 , j_1 \right\} \cup \left\{  j_2  \right\}  \cup \dots \cup \left\{ j_t \right\}$;
	thus it is a union of $t$ color classes of $f_{j_1}$.
	Since all the other color classes of $\mathcal{M}$ are some subsets of $[s-1] \setminus \left[ j_1 \right]$ and because every color class of $f_{j_1}$
	is either $\left\{  0 , j_1 \right\}$ or a singleton, it follows that each color class of $\mathcal{M}$ other than $\left\{ 0 , j_1 , j_2 , \dots , j_t \right\}$
	is also a union of some color classes of $f_{j_1}$.
	Accordingly, $\widetilde{f}_{j_1}$ is strictly finer than $\widetilde{\mathcal{M}}$.
	
	\noindent
	We conclude that $f_1 , f_2 , \dots , f_{s-1}$ are the lonely irreducible FAT colorings of $G$ subject to the condition that $\{0\}$ is not a color class.
	
	\noindent
	We summarize that, on one hand, $\Bigl\{  \{0\} , \{ 1 , 2 , \dots , s-1 \}  \Bigr\}$ is the unique irreducible FAT coloring of $G$ subject to the condition that $\{0\}$
	is a color class.
	On the other hand, for the case where $\{0\}$ is not a color class, the list of all irreducible FAT colorings is $f_1 , f_2 , \dots , f_{s-1}$.
	Accordingly, the graph $G$ has exactly $s$ irreducible FAT colorings.
\end{proof}

\vspace{0.4cm}

%%%%%%%%%%%%%%%%%%%%%%%%%%%%%%%%%%%%%%%%%%%%%%%%%%%%%%%%%%%%%%%%%%%%%%%%%%%%%%%%%%%%%%%%%%%%%%%%%%%%%%%%%%%%%%%%%%%%%%%%%%%%%%%%%%%%%%%%%%%%%%%%%%%%%%%

\def\cprime{$'$} \def\cprime{$'$}

%%%%%%%%%%%%%%%%%%%%%%%%%%%%%%%%%%%%%%%%%%%
%\def\cprime{$'$} \def\cprime{$'$}

%\end{thebibliography}

%\bibliographystyle{plain}
%\bibliography{MyReferences}
\end{document}